\documentclass[12pt]{article}

\usepackage{amsfonts}
\usepackage{amsmath}
\usepackage{amssymb}
\usepackage{amsthm}

\def\ra{\rightarrow}
\def\lra{\leftrightarrow}

\def \Z {\mathbb Z}

\def \R {\mathbb R}
\def \N {\mathbb N}
\def\P{\mathbb P}
\def\al{\alpha}
\def\et{\eta}
\def\th{\theta}
\def\thp{\theta^+}
\def\thsm{\theta^{- *}}
\def\Hsm{H^{- *}}

\def\be{\beta}
\def\ep{\varepsilon}
\def\de{\delta}

\def\la{\lambda}
\def\La{\Lambda}
\def\ga{\gamma}

\def\si{\sigma}
\def\de{{\delta}}

\def\om{\omega}
\def\Om{\Omega}
\def\muh{\mu^{(h)}}

\def\Puh{P^{(h)}}
\def\Ph1{P^{(h_1)}}
\def\Ph2{P^{(h_2)}}

\def\b0{{\bf 0}}
\def\bk{{\bf k}}
\def\prl{\parallel}

\newtheorem{thm}{Theorem}[section]

\newtheorem{lem}[thm]{Lemma}
\newtheorem{cor}[thm]{Corollary}

\theoremstyle{plain}

\newtheorem{defn}[thm]{Definition}

\begin{document}

\title{Approximate zero-one laws and sharpness of the percolation transition in a class of
models including 2D Ising percolation}

\author{J. van den Berg\footnote{Part of this research has been funded by the Dutch BSIK/BRICKS project.}, 
  \\
{\small CWI and VUA, Amsterdam, The Netherlands} \\
{\footnotesize email: J.van.den.Berg@cwi.nl}
}
\date{}
\maketitle

\begin{abstract}
One of the most well-known classical results for site percolation on the square lattice
is the equation
$p_c + p_c^* = 1$. In words, this equation means that
for all values $\neq p_c$ of the parameter $p$ the following holds:
Either a.s. there
is an infinite open cluster or a.s. there is an infinite closed `star' cluster.
This result is closely related 
to the percolation transition being sharp:
Below $p_c$ the size of the open cluster of a given vertex
is not only (a.s.) finite, but has a distrubtion with an exponential tail. 
The  analog of this result has been proved by Higuchi in 1993 for two-dimensional Ising percolation,
with fixed inverse temparature
$\beta <\beta_c$, and as parameter the 
external field $h$.

Using sharp-threshold results (approximate zero-one laws) and a modification of an RSW-like result by
Bollob\'as and Riordan, we show that these results hold for a large class of percolation models
where the vertex values can be `nicely' represented (in a sense which will be defined precisely)
by i.i.d. random variables.
We point out that 
the ordinary percolation model belongs obviously to this class, and we show that also the above mentioned
Ising model belongs to it. 

We hope that our results improve insight in the Ising percolation model, and will help
to show that many other (not yet analyzed) weakly dependent percolation models also belong to the above
mentioned class.

\end{abstract}

\begin{section}{Introduction}
A landmark in the development of percolation theory is the proof by Harry Kesten in 1980(\cite{Ke80})
that the critical probability for bond percolation on the square lattice equals $1/2$.
A key argument in his proof is what would now be called a 'sharp-threshold result': 
He showed that if $n$ is large and
the probabilty of having an open horizontal crossing of an $n \times n$ box is
neither close to $0$, nor close to $1$, 
there is a reasonable probability to have many (in fact, at least of order $\log n$)
so-called pivotal edges (or cut edges).
(These are edges $e$ with the property that changing the state (open/closed) of $e$,
changes the occurrence or non-occurrence of an open horizontal crossing).

The proof of this intermediate key result has a combinatorial-geometric
flavour: it involves a `counting argument', with conditioning on the lowest open crossing (and the left-most
closed dual crossing of the area above the just mentioned open crossing).
This result in turn implies that the derivative
(w.r.t. the parameter $p$) of the crossing probability is very large (also at least of order
$\log n$) if $n$ is very
large. 
Since probabilities are at most $1$, it is impossible to have such behaviour for all $p$ throughout some interval of
non-zero length.  On the other hand, other arguments show that the above mentioned crossing
probability is bounded away from $0$ and $1$, uniformly in $n$ and $p \in (1/2, p_c)$.
Hence $p_c$ must be equal to $1/2$. (It was already known
(\cite{Ha}) that $p_c \geq 1/2$).

Soon after Kesten's result it was shown (\cite{Ru81} and \cite{Wi}) that
his arguments can also be used to prove related long-standing
conjectures, in particular that $p_c + p_c^* = 1$ for site percolation on the square lattice. Here $p_c^*$
denotes the critical probability for the so-called matching lattice (or star lattice) of the square lattice:
this is the lattice with the same vertices
as the square lattice, but where each vertex $(x,y)$ has not only edges to its four horizontal or vertical
`neighbours' $\{(x',y') \, : \, |x-x'| + |y - y'| = 1\}$ but also to the nearest vertices in the diagonal directions, 
$\{(x',y') \, :\, |x - x'| = |y - y'| = 1$.

Russo (\cite{Ru82}) was the first who put the above mentioned `sharp-threshold' argument
of Kesten in a more general framework by 
formulating an {\it approximate zero-one law}.
This approximate zero-one law itself is {\it not} a
percolation result.
It is, as the name indicates, a 'finite' approximation of Kolmogorov's
zero-one law: 
Recall that the latter says (somewhat informally) that if $X_1, X_2, \cdots$
are i.i.d. Bernoulli random variables (say, with parameter $p$) and $A$ is an event with the property that
the occurrence or non-occurrence of $A$ can not be changed by changing a single $X_i$, then $A$ has
probability $0$ or $1$. Russo's approximate law says that if $A$ is an increasing event with the property
that for each $i$ and $p$ the probability that changing the state of $X_i$ disturbs the occurrence or
non-occurrence of $A$ is very small, then for all $p$, except on a very small interval, the probability
of $A$ is close to $0$ or close to $1$.

Given this approximate zero-one law, the combinatorial-geometric argument in Kesten's work discussed
above (to get a lower bound for the
(expected) number of pivotal items), can be (and was, in Russo's paper) replaced by the considerably
simpler (and `smoother') 
argument that the probability that a given edge (or, for site percolation, site) is pivotal is small
when $n$ is large. 

It should be noted that for the other, more standard part of the proof, the approximate
zero-one law does not help: A `separate' argument of the form that if the probability of an open crossing of a square is
sufficiently large,
there is an infinite open cluster, remains needed.
For this argument, which is often
called a `finite-size criterion', 
the RSW theorem (\cite{Ru78} and \cite{SeWe78}; see also \cite{Gr}, Chapter 11) is  essential.
Informally, this theorem gives a suitable lower bound for the probability of
having an open crossing (in the `long' direction) of a $2 n \times n$ rectangle, in terms of the probability
of having a crossing of an $n \times n$ square. The classical proof uses conditioning on the lowest crossing.
For ordinary Bernoulli percolation this works fine, but, as remarked earlier, in dependent models such conditioning
often leads to very serious, if not unsolvable, problems. An important recent achievement in this respect is
a `box-crossing' theorem obtained in \cite{BoRi2}, of which the proof is much more robust than that of the `classical'
RSW result and does not use such conditioning. We will come back to this later.

Refinements, generalizations, and independent results with partly the same flavour as Russo's approximate
zero-one law have been obtained and/or applied by Kahn, Kalai and Linial
(\cite{KaKaLi}), Talagrand (\cite{Ta}), Friedgut and Kalai (\cite{FrKa}), Bollob\'as and Riordan 
(see e.g. \cite{BoRi1}), Graham and Grimmett
(\cite{GrGr}), Rossignol (\cite{Ro}) and others, and have become known under the name `sharp threshold results'.

Although such results take Kesten's key argument
in a more
general context, involving `less geometry', Kesten's proof is essentially still the shortest and, from a 
probabilistic point of view, intuitively most appealing, self-contained proof of $p_c = 1/2$ for bond
percolation on the square lattice:
none of the above-mentioned general sharp-threshold theorems has a short or probabilistically intuitive
proof. Moreover, the combinatorial-geometric ideas and techniques in Kesten's proof have turned
out to be very
fruitful in other situations, for instance in the proof of one of the main results in
Kestens paper on scaling relations
for 2D percolation (\cite{Ke87}).

On the other hand, there are examples of percolation models where
Kesten's argument is difficult and painful to carry
out, or where it is even not (yet) known how to do this. An axample of
the latter is the Voronoi percolation model, for which Bollob\'as and Riordan (\cite{BoRi2}; see also
\cite{BoRi3})) proved (using
a sharp threshold result from \cite{FrKa}) that
it has critical probability $1/2$. This had been conjectured for a long time, but carrying out Kesten's
strategy for that model led to (so far) unsolved problems. 

An example of the former is percolation of $+$ spins in
the 2D Ising model with
fixed inverse temperature $\beta < \beta_c$, and with external field parameter $h$ (which plays the role of $p$ in
'ordinary' percolation). Higuchi (\cite{Hi1} and \cite{Hi2}) showed that for all values of $h$,
except the critical value
$h_c$, either (a.s.) there is an infinite cluster of vertices with spin $+$, or (a.s) there is an 
infinite * cluster (that is, a cluster in the * lattice) of vertices with spin $-$. (The result is stated
in \cite{Hi2} but much of the work needed in the proof is done in \cite{Hi1}). Higuchi followed globally 
Kesten's arguments. However, to carry them out (in particular the `counting under conditioning on the lowest
crossing' etc.) he had to overcome several new technical difficulties due to the dependencies in this model.
This makes the proof far from easy to read.

\smallskip\noindent
{\bf Remark:}
It should be mentioned here that there is also a very different proof of $p_c = 1/2$ for bond percolation on
$\Z^2$ (and
$p_c + p_c^* = 1$ for site percolation), namely by
using the work of Menshikov (\cite{Me}) and of Aizenman and Barsky (\cite{AiBa}).
They gave
a more `direct' (not meaning `short' or `simple') proof, without using the results or arguments indicated above,
that for every $d \geq 1$ the
cluster radius distribution for independent percolation on $\Z^d$ with $p < p_c$ has an exponential tail.
However, their proofs make crucially use of the BK inequality (\cite{BeKe}), and
since our interest is mainly in dependent percolation models (for which no suitable analog of this
inequality seems to be available), these proof methods will no be discussed in more detail here.

\smallskip
In the current paper we present a theorem (Theorem \ref{mainthm}) which
says that the analog of $p_c + p_c* = 1$ holds for a large
class of weakly dependent 2D percolation models.  Roughly speaking, this class consists of sytems that
have a proper, monotone, `finitary' representation in terms of i.i.d. random variables.
From the precise definitions it
will be immediately clear that it contains the ordinary (Bernoulli) percolation models. We give,
(using results obtained
in the early nineties by Martinelli and Olivieri (\cite{MaOl}), and modifications of results in (\cite{BeSt})
which were partly inspired by \cite{PrWi}), a `construction' of the earlier mentioned 2D Ising model
which shows that this model also belongs to this class.
We hope that this will improve our understanding of 2D Ising percolation,
and that it will help to find alternative constructions of other natural percolation models, which bring
them in the context of our theorem.

\smallskip\noindent
The theorem is based on: \\
(a) One of the sharp-threshold
results mentioned above, namely 
Corollaries 1.2 and 1.3 in \cite{Ta} , which are close in spirit to, but quantitatively more explicit than,
Russo's
approximate zero-one law. The reason for using the results in \cite{Ta} rather than those in \cite{FrKa} (which,
as said above, were applied to percolation problems by Bollob\'as and Rirordan) is that the latter assume
certain symmetry properties on the events to which they are applied. In many situations this
causes no essential difficulties, but it gives much more flexibility to allow an absence of
such symmetries (see the Remark following propery (iv) near the end of Section 2.1). \\
(b) A modification/improvement (obtained in \cite{BeBrVa}), of an RSW-like box-crossing theorem
of Bollob\'as and Riordan (\cite{BoRi2}). As indicated in the short discussion of Russo's paper above, some form
of RSW theorem seems unavoidable. For many dependent percolation  models it is very hard (or maybe impossible)
to carry out the original proof of RSW. The Bollob\'as-Riordan form of RSW (and its modification in \cite{BeBrVa}) is weaker
(but still strong enough) and much more robust with respect to spatial dependencies.

\smallskip
In Section 2 we introduce some terminology and state Theorem \ref{mainthm}, which says that a large
class of 2D percolation models satisfies an analog of $p_c + p_c^* = 1$. We also state some
consequences/examples of the theorem
In particular we show that the Ising percolation result by Higuchi satisfies the conditions of 
Theorem \ref{mainthm}.

In Section 3 we state preliminaries needed in the proof of Theorem \ref{mainthm}: Talagrand's result
mentioned above, and an extension of his result to the case where the underlying random variables
can take more than two different values, and where the events under consideration not necessarily
depend on only finitely many of these underlying variables.
In that Section we also explain that the earlier mentioned (modification of the)
RSW-result of Bollob\'as and Riordan applies to our class of percolation models, and we prove other properties
that are used in the proof of Theorem \ref{mainthm}.
 
In Section 4 we finally prove Theorem \ref{mainthm}, using the ingredients explained in Section 3.

Apart from the proofs of the RSW-like theorem and of Talagrand's sharp threshold result mentioned above,
the proof of Theorem \ref{mainthm} is practically self-contained.

\end{section}

\begin{section}{Statement of the main theorem and some corollaries}
\begin{subsection}{Terminology and set-up} \label{term}
In this subsection we will describe the (dependent) percolation models on the square lattice for
which our main result, a generalization of the well-known $p_c + p_c^* = 1$ for ordinary percolation, holds.

\smallskip\noindent
Let $k$ be a positive integer and let $\mu^{(h)}, \, h \in \R$ be a family of probability measures on
$\{0, 1, \cdots, k\}$, indexed by the parameter $h$, with the following two properties ((a) and (b)): \\

\begin{eqnarray} \nonumber
\mbox{ (a) } & & \mbox{For each } 1 \leq j \leq k, \muh(\{j, \cdots, k\}) \mbox{ is a
continuously differentiable }, \\ \nonumber
\,& & \mbox{ strictly increasing function of } h. \\ \nonumber
\mbox{ (b) } & & \lim_{h \ra \infty} \muh(k) = \lim_{h \ra -\infty}\muh(0) = 1. \nonumber
\end{eqnarray}

Let $I$ be a countable set. Before we go on we need some notation and a
definition: We use '$\subset \subset$' to indicate `finite subset of'. 
The special elements $(0, 0, 0, \cdots)$ and $(k, k, k, \cdots)$ of 
$\{0, \cdots,k\}^I$ are denoted by $\b0$ and $\bk$ respectively.

Let $f \,:\, \{0, \cdots,k\}^I \ra \R$ be
a function.
Let $V \subset \subset I$ and let $y \in \{0, \cdots,k\}^V$.
For $x \in \{0, \cdots,k\}^I$,
we write $x_V$ for the `tuple' $(x_i \, i\in V)$.
We say that $y$ determines (the value of) $f$ if
$f(x) = f(x')$ for all $x, x'$ with $x_V = x'_V = y$.

\smallskip\noindent
Let $X_i, i \in I$ be independent random variables, each with distribution $\muh$. Let $\Puh$ denote
the joint distribution of the $X_i$'s. The $X_i$'s will be
the `underlying' i.i.d. random variables for our percolation system. We will assume that
the
`actual spin variables', which take values $+1$ (`open') and $-1$ (`closed') and which will be denoted
by $\si_v, \, v \in \Z^2$ below, are `suitably described'
in terms of the underlying $X$ variables: For each $v \in \Z^2$, its spin variable $\si_v$ is a function of the
$(X_i, i \in I)$. These functions themselves do not depend on $h$, but changing $h$ will change the
distribution of the $X$ variables and thus that of the $\si$ variables. More precisely, we assume that
$\si_v, v \in \Z^2$ are random variables with
the following properties ((i) - (iv) below):

\begin{itemize}
\item (i) ({\it Monotonicity.}) For each $v$, $\si_v$ is a measurable, increasing, $\{-1, +1\}$-valued
function of the collection
$(X_i, \, i \in I)$. Moreover, for each $v \in \Z^2$,  $\si_v(\b0) = -1$ and $\si_v(\bk) = +1$.

\item (ii) ({\it Finitary representation.}) There are $C_0 > 0$ and $\ga > 0$ such that for each 
$v \in \Z^2$ there is a sequence $i_1(v), i_2(v), \cdots$ of elements of $I$ such that
for all positive integers $m$, and all $h \in \R$,

$$\Puh\left((X_{i_1(v)}, \cdots, X_{i_m(v)}) \mbox{ does not determine } \si_v\right) 
\leq 
\frac{C_0}{m^{2 + \ga}}.
$$

\item
(iii) ({\it Mixing}. \\
$$ \exists \al > 0 \,\, \forall v , w \in \Z^2 \,\, \forall m < \al ||v-w||, \, \,
\{i_1(v), \cdots, i_m(v)\} \cap \{i_1(w), \cdots, i_m(w)\} = \emptyset.$$

\item (iv) For each $h$, the distribution of $(\si_v, v \in \Z^2)$ is translation invariant and invariant under
rotations over $90$ degrees, and under vertical and horizontal axis reflection.
\end{itemize}

{\bf Remark:} Note that, in property (iv), we do not require that
we can identify $I$ with $\Z^2$ in such a way that there is a stationary mapping (that is, a mapping which commutes
with shifts) from the process
$(X_i, i \in \Z^2)$ to the process $(\si_v, v \in \Z^2)$. In many cases there will be such identification,
but we found its requirement unnecessarily strong for our purposes (see for instance the example of the
Ising model below, where the mapping under consideration is not of this form).
A consequence of the absence of this requirement is an absence of certain symmetries needed to apply the
sharp thereshold results in, e.g. \cite{FrKa}. This is the main reason for using the results in \cite{Ta}.
\smallskip\noindent
\begin{defn}\label{maindef}
If a random field $(\si_v, v \in \Z^2)$ has the properties (i) - (iv) above, we say that the process has a {\it
nice, finitary} representation (in terms of the $X$ process, and with parameter $h \in \R$).
\end{defn}

\end{subsection}

\begin{subsection}{Statement of the main theorem and some special cases}
\smallskip\noindent
Now we consider percolation in terms of the $\si$ variables: We interpret $\si_v=+1$ ($-1$) as the vertex $v$ being
open (closed) and are interested in (among other things) the existence of infinite paths on which every
vertex is open. As usual, in our notion of `ordinary' paths we allow only horizontal and vertical steps, and
we use the term star paths when, in addition to these steps, also diagonal steps are allowed. Similarly (and again
following the usual conventions), we define `ordinary' clusters as well as star clusters. When we speak
simply of `cluster', we will always mean an `ordinary' cluster.

The $+$ cluster of a vertex $v$ will be denoted by $C_v^+$; the $- *$ cluster of $v$ (that is, the $-$ cluster of $v$
in the star lattice) will be denoted by $C_v^{- *}$,
etc. If $v = 0$ (the vertex $(0,0)$) we will often
omit the subscript $v$.

Recall that $\Puh$ denotes the probability distribution of the collection $(X_i, \, i \in I)$. We will
also use it for
the probability measure on $\{-1, +1\}^{\Z^2}$ induced by the map from the $X$ variables to the $\si$ variables.
Since the context in which it is used will always be clear, this should not cause any confusion.

\begin{thm}\label{mainthm}
Let $(\si_v, v \in \Z^2$) be a spin system with  a nice, finitary representation, with parameter
$h \in \R$ (in the sense of Definition \ref{maindef}).\\
Then there is a {\it critical value} $h_c$ of $h$ such that:  

\smallskip\noindent
(a) $\forall h > h_c \,\,\Puh(|C^+| = \infty) > 0$ and the distribution of $|C^{- *}|$ has
an exponential tail. \\
(b)$\forall h < h_c \,\, \Puh(|C^{- *}| = \infty) > 0$ and the distribution of $|C^+|$
has an exponential tail. 

\end{thm}

{\bf Remark:} Note that it follows from the statement of the theorem that $h_c$ satisfies
$$h_c = \inf\{h : \Puh(|C^+| = \infty) > 0\} = \sup\{h : \Puh(|C^{- *}| = \infty) > 0\}.$$
Also note that if reversal of $h$ corresponds with a spin-flip (more precisely, if
the distribution of $\si$ under $\Puh$ is the same as the distribution of $-\si$ 
($ = (\si_v, v \in \Z^2)$) under
$P^{(-h)}$, the above theorem immediately implies

\begin{equation}
\label{mat-rel}
h_c + h^*_c = 0, \mbox{ where }
\end{equation}

$$h^*_c = \inf\{h : \Puh(|C^{+\, *}| = \infty) > 0\}.$$

\begin{subsubsection}{Special cases} \label{spc}
{\bf Bernoulli site percolation on the square lattice, with parameter $p$.} \\
This model, where the vertices are open ($+1$) with probability $p$ and closed ($-1$) with probability
$1-p$ trivially satisfies the conditions of Theorem \ref{mainthm}: Simply take $I = \Z^2$, $k=1$ (that is,
the $X_i$'s take values $0$ and $1$), and $\si_v = 2 X_v - 1$, $v \in \Z^2$. Finally, take for instance
(note that we want $\muh(1)$ to go $1$ (respectively $0$) as $h \ra \infty$ ($-\infty$))

$$\muh(1) = \frac{\exp(h)}{\exp(h) + \exp(-h)}.$$

Taking $p = \muh(1)$ completes the 'translation'. It is easy to see that reversing $h$ corresponds 
with a spin-flip, so that \eqref{mat-rel} holds, which is equivalent to the well-known

$$p_c + p_c^* = 1,$$ 

for this model.

\medskip\noindent
{\bf Models defined explicitly in terms of i.i.d. random variables} \\
In the previous example, the representation in terms of i.i.d. random variables was explicit and trivial.
It is easy to find many other examples with explicit (but less trivial) representations.
For instance, take $I = \Z^2$, and let the $X$ variables be i.i.d. Bernoulli with parameter $p$. Define,
for each $v \in \Z^2$, $\si_v$ as follows: Consider, for each $n$, the difference between the number of $1$'s
and the number of $0$'s in the $2 n \times 2 n$ square centered at $v$. Take the smallest $n$ where this difference
has absolute value larger than some constant, say $5$. Define $\si_v$ as the
sign of the above mentioned difference (number of $1$'s minus
number of $0$'s) for that $n$. It is easy to check that this definition corresponds with a nice, finitary
representation in the sense of Definition \ref{maindef}. More interesting (in the context of the subject of this paper)
are those weakly dependent models that are not apriori explicitly defined in terms of such 
representation. One can then search for a possible `hidden' representation. A major example where this works is
the following.

\medskip\noindent
{\bf Ising model with (fixed) inverse temperature $\be < \be_c$ and external field parameter $h$} \\
We first recall some definitions and standard results for these models.
Ising measures $\mu_{\be,h}$ on $\{-1,+1\}^{\Z^2}$, with inverse temperature $\beta \in [0, \infty)$ and external field
$h \in (-\infty, \infty)$ are  probability
measures that satisfy, for $\et \in \{0,1\}$ and $v \in \Z^2$,

\begin{equation}
\label{Is-def}
\mu_{\be,h}(\si_v = \et \, | \, \si_w, \, w \neq v) = \frac{\exp\left(\be \et (h + \sum_{w \sim v} \si_w)\right)}
{\exp\left(\be \et (h + \sum_{w \sim v} \si_w)\right) + \exp\left(- \be \et (h + \sum_{w \sim v} \si_w)\right)},
\end{equation}  
where $w \sim v$ means that $||v-w|| := |v_1 - w_1| + |v_2 - w_2| =1$.

It is well-known that there is a critical value $\be_c$ such that for $\be < \be_c$ there is a unique measure
satisfying \eqref{Is-def}, while for $\be > \be_c$ there is more than one such measure.

The Ising model is one of the most well-known examples of a Markov random field:
the conditional distribution of the spin value of a vertex $v$, given the
spin values of all other vertices, depends only on the spin values of the neighbours of $v$. 

The `single-site' conditional distributions in \eqref{Is-def} will be used often in the remainder of this
subsection and will be denoted by $q_v^{\al}$. More precisely, let for $v \in \Z^2$, $\partial v$ denote the set
of (four) vertices that are neighbours of $v$. Further, for $\al \in \{-1,+1\}^{Z^2}$ and $V \subset \subset \Z^2$,
let $\al_V$ denote the 'restriction' of $\al$ to $V$; that is,
$\al_V = (\al_w, w \in V)$.
For $\al \in \{-1, +1\}^{\partial v}$ and $\et \in \{-1, +1\}$ we
define $q_v^{\al}(\et)$ as the conditional probability that $\si_v$ equals $+1$ given that $\si_{\partial v}$ equals
$\al$: 
\begin{equation}
\label{q-def}
q_v^{\al}(\et) = q_v^{\al}(\et; \be, h) := \frac{\exp\left(\be \et(h + \sum_{w \sim v}\al_w)\right)}
{\exp\left(\be \et(h + \sum_{w \sim v}\al_w)\right) + \exp\left(-\be \et(h + \sum_{w \sim v}\al_w)\right)}.
\end{equation}
Note that the dependence on the `neighbour configuration' $\al$ is only through the {\it number} of $+$ (and hence
of $-$) spins in $\al$.
Therefore it is also convenient to define, for $m = 0, \cdots, 4$,
\begin{equation} \label{qdef2}
q_v^{(m)}(\et) = q_v^{\al}(\et),
\end{equation}
where $\al$ may be any element of $\{-1, +1\}^{\partial v}$ with the property that the number of $w \sim v$ with
$\al_w = +1$ equals $m$.


The following result is well-known and goes back to \cite{AiBaFe} and \cite{Le}.
Higuchi (\cite{Hi1}
proved and used a stronger result but the weaker version below is sufficient for our purpose.

\begin{thm} \label{mixing}
There exist $C_1 > 0$ and $\ep > 0$ (which depend on $\be$ but not on $h$) such that

\begin{equation}
\label{mix-prop}
\mu_{\be,h}(\si_0 = +1 \, | \, \si_{\partial \La(n)} \equiv +1) - 
\mu_{\be,h}(\si_0 = +1 \, | \, \si_{\partial \La(n)} \equiv -1)  
\leq C_1 \, \exp(-\ep n).
\end{equation}
Here $\La(n)$ denotes the set of vertices $[-n,n]^2$ and $\partial \La(n)$ the boundary of this set.
\end{thm}

Martinelli and Olivieri (Theorem 3.1 in \cite{MaOl}) have proved for a large class of spin systems that such
a {\it spatial} mixing property implies exponential convergence (to equlibrium) for certain dynamics. For the
Ising model this dynamics is as follows. First we define the notion {\it local update}.
Let $\al \in \{-1, +1\}^{\Z^2}$ and $v \in \Z^2$. By a local update of the spin value of $v$ (in the configuration $\al$)
we mean that we draw a new value, say $\et$,  according to the distribution $q_v^{\al_{\partial v}}(\cdot)$,
and leave $\al$
unchanged everywhere except at $v$ where we replace $\al_v$ by $\et$. The dynamics can now be described as follows:
Start from some initial configuration. Each vertex is {\it activated} at rate $1$.
When a vertex is activated, a local update at that vertex is made.
The Martinelli-Olivieri result (for the special case of the Ising model) says that the distribution at time $t$, starting
from any initial configuration, converges exponentially fast (uniformly in $h$) to $\mu_{\be,h}$.

As observed in
(\cite{BeSt}), this also holds for certain discrete-time versions of the dynamics. 
The discrete-time
dynamics in (\cite{BeSt}) involves auxiliary random
variables in terms of which the dynamics is not monotone.
For the purpose in (\cite{BeSt}) that did not matter, but 
this dynamics is not suitable for our current purpose:
to `construct' the Ising measure in such a way that it
fits with Definition \ref{maindef}.
The following dynamics {\it is} suitable for our purpose, and the Martinelli-Olivieri proof (with straightforward
modifications) works for this dynamics as well: In this discrete-time dynamics we update all even vertices at the
even times and all odd vertices at the odd times. (A vertex is even (odd) if the sum of its coordinates is even (odd)).
Note that these `parallel' updates are well-defined, since the update of an even (odd) vertex only involves the `current'
spin values of its neighbours, each of which is odd (even). 

To describe the Ising model 
as a nice, finitary representation in the sense of Definition \ref{maindef},
we describe these local updates as follows in terms of i.i.d. random variables $Y_i(t), i \in \Z^2,
t \in \N$, which take values in $\{0, 1, \cdots, 4\}$. Here (and further) $\si_v^{\om}(t)$ denotes
the spin value at vertex $v$ at time $t$,
for the system starting at time $0$ with configuration $\om$. Sometimes we will omit the superscript
$\om$. At each even time $t$ we do the following,
for each even vertex $v$: If the number of $w \sim v$ with $\si_t(w) = -1$ is at most $Y_v(t)$ we set $\si_v(t+1) := +1$,
otherwise we set $\si_v(t+1) := -1$. For odd $t$ we do the analogous actions for all odd $v$. 
It is easy to see (recall \eqref{qdef2}) that if we take the following
distribution for the $Y$ variables, these actions correspond
exactly with the earlier defined notion of local updates:

$$P(Y_v(t) \geq m) = q^{(4-m)}_0(+1; \be,h), \,\, 0 \leq m \leq 4.$$

An advantage of such auxiliary variables is that it enables to couple systems starting from different initial
configurations. Define $\si^{\om}(t) = (\si_v^{\om}(t), \, v \in \Z^2)$ as the configuration at
time $t$ for the system that
starts at time $0$ with configuration $\om$ and follows the above mentioned dynamics (involving the $Y$ variables).
We will simply
replace the superscript $\om$ by $+$ when we start with the initial configuration where each vertex has
value $+1$, and by $-$ when we start with $-$ values.
As said before, the Martinelli-Olivieri result extends to this dynamics.

In terms of the above notation, the Martinelli-Olivieri result tells us that
there are positive $C_2$ and $\la_2$ (which depend on $\be$ but not on $h$) such that for all $t$

\begin{equation}
\label{mo}
P(\si_v^+(t) \neq \si_v^-(t)) \leq C_2 \exp(-\la_2 t).
\end{equation}

Also note that we can extend the collection of $Y$ variables to negative $t$, and that for
all integers $s, t$ with $s \leq t$, and all configurations $\om \in \{-1,+1\}^{\Z^2}$, we can define
$\si^{\om}(s,t) = (\si_v^{\om}(s,t), \, v \in \Z^2)$ as the configuarion at time $t$ for the system that
starts at time $s$ with configuration $\om$ and evolves as described above.
Analogously as in \cite{BeSt}  (which was partly inspired by the perfect-simulation ideas in \cite{PrWi}),
we observe that if $t < 0$ and $\si_v^{+}(t, 0) = \si_v^-(t,0)$, then (by obvious monotonicity)
$\si_v^{\om}(s,0) = \si_v^{\om'}(s,0)$ for all $s \leq t$ and all $\om, \om'$.
From this observation, \eqref{mo} and standard arguments it follows that if we define

$$\tau(v) = \max\{t < 0 \, : \, \si_v^{+}(t, 0) = \si_v^-(t,0)\}, \,\, v \in \Z^2,$$

and

\begin{equation}
\label{sig-def}
\si(v) = \si_v^+(\tau(v),0) \,\,(= \si_v^-(\tau(v),0)), \,\, v \in \Z^2,
\end{equation}
we have that $\si := (\si(v), v \in \Z^2)$ has the Ising distribution $\mu_{\be,h}$, and that

\begin{equation}
\label{tau-exp}
P(\tau(v) \geq n) \leq  C_2 \exp(-\la_2 n).
\end{equation}

This shows that the Ising distribution
has indeed a nice, finitary representation (in the sense of Definition \ref{maindef}):
Take $I = \{(v,t) \, : \, v \in \Z^2, t \in \Z, t < 0\}$, and 
$X_{(v,t)} = Y_v(t), \, (v,t) \in I$. Then (i) is clear. To see (ii), note that, for each $t < 0$,  $\si_v^+(t,0)$ and
$\si_v^-(t,0)$ are completely determined by the variables $Y_w(s), \, t \leq s < 0, \, \parallel w-v \parallel < s$. \\
So for the sequence $i_1(v), i_2(v), \cdots$ we can take $(v, -1)$, followed by an enumeration
of the (finite) set $\{(w,-2)\, : \, w \in \Z^2, \prl w-v\prl < 2\}$, followed by an enumeration of
$\{(w,-3)\, : \, w \in \Z^2, \prl w-v \prl < 3\}$ etc. 
The upper bound in (ii) (in fact, even a stronger bound) for
the probability that $X_{i_1(v)}, \cdots X_{i_m(v)}$ does not determine $\si_v$
follows from \eqref{tau-exp} and the fact that the set
$\{(w,s) \, : \, \prl w-v \prl  < |s|, \, t \leq  s < 0\}$
has of order $|t|^3$ elements.
Property (iii) is now also clear. Property (iv) is standard (and has nothing to do with the above description
of the Ising model in terms of the $Y$ variables: Since $\beta < \beta_c$, there is a unique Ising measure
with parameters $\beta, h$ and this measure inherits the symmetry properties in the definition of the model). 

Hence we may apply Theorem \ref{mainthm}.
Moreover, the spin-flip symmetry mentioned in
the remark around \eqref{mat-rel} is clearly satisfied.
So we get the following, which is 
the earlier mentioned result by Higuchi (see Theorem 1 (and Corollary 2) in \cite{Hi2}).

\begin{thm} \label{Hi-thm}
Let $\be < \be_c$ and consider the Ising measures $\mu_{\beta,h}, \, h \in \R$ on the square lattice.
Statements (a) and (b) of Theorem
\ref{mainthm} above (with $\Puh  = \mu_{\be, h}$),
as well as equation \eqref{mat-rel}, hold for this model.
\end{thm}
{\bf Remarks:} 
(i) The sharp-threshold result in \cite{GrGr} may provide yet another route to prove this result for
the Ising model.
However, that sharp-threshold result is not suitable for the proof of our general Theorem \ref{mainthm}, because
the random field $\si_i, i \in \Z^2$ in Theorem \ref{mainthm} does not necessarily satisfy the strong
FKG condition needed in \cite{GrGr}. \\
(ii) We hope that, like the Ising model, many other models which at first sight are not covered by
Theorem \ref{mainthm}, can be constructed or represented in such a way that this Theorem does apply.
However, we do not claim that this Theorem gives a completely general recipee. For instance, attempts
to bring the models treated
in \cite{BaCaMe} (which have some of the flavour of the Ising model) in the context of this theorem have, so
far, not been successful.

\end{subsubsection}
\end{subsection}
\end{section}
\begin{section}{Preliminaries}
\begin{subsection}{Approximate zero-one laws} \label{sect-azol}
A key ingredient in our proof of Theorem \ref{mainthm} is a sharp threshold result (or approximate zero-one 
law). As we said in Section 1, there are several of such results in the literature. 
The one we use is Corollary 1.2 in Talagrand's paper \cite{Ta}, which
is somewhat similar in spirit to Russo's approximate
zero-one law (\cite{Ru82}), but more (quantitatively) explicit.


These threshold results are, although particularly useful for percolation, of a much more
general nature.
Consider the set $\Om := \{0,1\}^n$, which for the moment serves as our sample space. 
For $\om$, $\om' \in \Om$ we say that $\om \leq \om'$ (or, equivalently, $\om' \geq \om$)
if $\om_i \leq \om'_i$ for all $1 \leq i \leq n$.
Following the standard terminology we say that an event (subzet of $\Om$) is increasing if
for each $\om \in A$ and each $\om' \geq \om$ we have $\om' \in A$.
For $\om \in \Om$ and $1 \leq i \leq n$ we define $\om^{(i)}$ as the configuration obtained
from $\om$ by flipping $\om_i$. More precisely, $\om^{(i)}_j$ is equal to $\om_j$ for 
$j \neq i$, and $1 - \om_j$ if $j = i$.

Let $A$ be an increasing event, $\om \in \Om$, and let $1 \leq i \leq n$.
We say that $i$ is an internal
pivotal index (for $A$, in the configuration $\om$) if $\om \in A$ but $\om^{(i)} \not \in A$.
It is easy to see from the fact that $A$ is increasing that this implies that $\om_i = 1$.

By $A_i$ we denote the event that $i$ is an internal pivotal for $A$; that is,

$$A_i = \{\om \, : \, \om \in A \mbox{ but } \om^{(i)} \not \in A\}.$$

Let, for $p \in (0,1)$,  $\P_p$ be the product measure with parameter $p$.
Talagrand's result to which we referred above is the following:

\begin{thm}\label{tal1.2} ({\em Talagrand (\cite{Ta}, Corollary 1.2)})
There is a universal constant $K_1$ such that for all $n$, all increasing events 
$A \subset \{0,1\}^n$ and all $p$,

\begin{equation} \label{t1.2-eq}
\frac{d}{d p} \, \P_p(A) \geq \frac{\log(1/\ep)}{K_1} \P_p(A) (1 - \P_p(A)),
\end{equation}
where $\ep = \ep(p) = \sup_{i \leq n} \P_p(A_i)$.
\end{thm}
{\bf Remark:} In fact, Corollary 1.2 in \cite{Ta} is somewhat sharper, namely with $K_1$ above
replaced by
$K p (1-p) \log[(2 / (p (1-p))]$, where $K$ is also a universal constant.
Since $p(1-p) \log[(2 / (p (1-p))]$ is bounded from above, Theorem \ref{tal1.2} above follows
immediately.

\smallskip\noindent
Let $p_1 < p_2$.
Noting, (as in Section 3 of \cite{Ta}) that \eqref{t1.2-eq} is equivalent to
$$\frac{d}{d p} \log \left(\frac{\P_p(A)}{1 - \P_p(A)} \right) \geq \frac{\log(1/\ep)}{K_1},$$
and integrating 
this inequality over the interval $(p_1, p_2)$, gives

\begin{cor} \label{tal1.3} (Talagrand (\cite{Ta}, Corollary 1.3))
There is a universal constant $K_1$ such that for all $n$, all increasing events  
$A \subset \{0,1\}^n$ and all $p_1 < p_2$,

\begin{equation}
\P_{p_1}(A) (1 - \P_{p_2}(A)) \leq (\ep')^{(p_2 - p_1)/K_1},
\end{equation}
where
\begin{equation} \label{eppdef}
\ep' = \sup_{p_1 \leq p \leq p_2} \max_{1 \leq i \leq n} \P_p(A_i).
\end{equation}
\end{cor}
{\bf Remark:} In the definition of $\ep'$ in the statement of Corollary 1.3 in \cite{Ta} the supremum involving
$p$ is over the interval $[0,1]$ instead of $[p_1, p_2]$, but it is clear that
the result with $\ep'$ as defined in \eqref{eppdef} holds.

\smallskip
We want to apply the above results to the family of distributions $\Puh, h \in \R$ in the statement
of Theorem \ref{mainthm}. Recall that $\Puh$ is the product over $I$ of the distribution of $\muh$, and that
the latter is a probability distribution on $\{0, \cdots, k\}$. First we have to `generalize' some of our
definitions.

The notion of increasing event is extended in the obvious way.
The extension of the notion of being {\it pivotal} is somewhat less obvious.
Let $A \subset \{0,1, \cdots, k\}^I$ be
an increasing event.
We say that index $i \in I$ is an internal pivotal
index (in a configuration $\om \in \{0,\cdots,k\}^I$ and for a given increasing event $A$) if
$\om \in A$ but $\om^{(i)} \not \in A$, where now $\om^{(i)}$ is defined as the configuration $\om'$ which has 
$\om'_j = \om_j$ for all $j \neq i$ and $\om'_i = 0$. (It follows immediately from the definition that if
$i$ is pivotal, then $\om_i > 0$).

We can not immediately use Corollary \ref{tal1.3}
because of the following two issues: One is that $k$ may be larger than $1$, the other is that $I$ is not finite but
countably infinite. In order to (partly) handle the latter, we first define
\begin{defn} \label{def-appr} 
(a) An increasing event $A \subset \{0, \cdots,k\}^I$ which is  defined in terms of finitely many coordinates
is called an {\it increasing cylinder event}. \\
(b) We say that an increasing event $A \subset \{0, \cdots, k\}^I$ is {\em strongly approximable} (w.r.t. the
probability measures $\Puh, h \in \R$) if there
is a sequence of  increasing cylinder events $A(n), n= 1, 2, \cdots$, such that 

$$\Puh(A(n) \triangle A) \rightarrow 0, \,\,\mbox{ as } n \rightarrow \infty,$$
for every $h \in \R$. In that case we say that $A$ is strongly approximated by the sequence $(A(n))$.
\end{defn}
{\bf Remark}: Clearly, if there are increasing cylinder events $A(n)$ 
such that $A(n) \rightarrow A$ in the set-theoretical sense, then $A$ is
strongly approximable.

We will use the following extension of Corollary \ref{tal1.3}.

\begin{cor}
\label{tal1.3xx}
If the event $A \subset \{0, 1, \cdots,k\}^I$ is increasing
and strongly approximable, then, for all $-\infty < h_1 < h_2 < \infty$,

\begin{equation}
\label{eq-tal1.3xx}
P^{(h_1)}(A) (1 - P^{(h_2)}(A)) \leq \bar\ep^{\frac{(h_2 - h_1) c(h_1, h_2)}{K2}},
\end{equation}
where $K_2$ is a constant,  $\bar\ep = \sup_i \sup_{h \in (h_1, h_2)} \Puh(A_i)$ and
$$c(h_1, h_2) = \inf_{h \in [h_1, h_2]} \min_{1 \leq j \leq k} \frac{d}{d h} \muh(\{j, \cdots, k\}).$$
\end{cor}
{\bf Remark:} This Corollary can be proved by suitably decoding $\{0, \cdots,k\}$ valued random
variables in terms of i.i.d. $\{0, 1\}$ valued random variables, and by careful approximation
arguments. However, such proof is rather cumbersome, and R. Rossignol (private communication) has pointed out
that it is possible to prove
Corollary \ref{tal1.3xx} in a more direct way by:
(a) Proving, for all positive $k$ `at once', an extension to spaces $\{0, \cdots, k\}^n$
of Theorem 1.5 in \cite{Ta}. (That theorem is a `functional' generalization of Theorem 1.1 in \cite{Ta}, of which
Theorem  \ref{tal1.2} above is an easy consequence). (b) Generalize the extension in (a) to countable product spaces by
`conditioning on the first $m$ coordinates, and then letting $m \rightarrow \infty$'. His proof even works for all
increasing events $A \subset \{0, 1, \cdots,k\}^I$. His arguments will appear in \cite{Ro07b}.

\end{subsection}
\begin{subsection}{Mixing property} \label{mix}
In this subsection we show that random variables $\si_v, v \in Z^2$ that
satisfy properties (i)-(iv) in Subsection 2.1, have certain very convenient spatial
mixing properties. \\
We say that a vertex $v$ is $l$ {\it determined} (w.r.t. the $X$ configuration)
if $X_{i_1}(v), \cdots, X_{i_l}(v)$ determine $\si_v$. A set of vertices $W$ is 
said to be $l$
determined if every $v\in W$ is $l$ determined.
From property (ii) in Section 2.1 we have

\begin{eqnarray}
\label{lconc}
\Puh(W \mbox{ not } l \mbox{ determined }) & \leq &
|W| \, \max_{v \in W} \Puh(v \mbox{ not } l \mbox{ determined }) \\ \nonumber
\, & \leq & |W| \frac{C_0}{l^{2 + \ga}},
\end{eqnarray}
with $C_0$ as in property (ii).

\begin{lem} \label{mix-lem} 
Let $k$ be a positive integer and let $U$ and $V$ be finite subsets of $\Z^2$ that have
distance larger than $k$ to each other. Let $A$ be an event that is defined in terms of the
random variables $\si_v, v \in U$ and $B$ an event that is defined in terms of the
random variables $\si_v, v \in V$. Then, with $\al$ and $\gamma$ as in properties (ii) and (iii)
in section 2.1,

\begin{equation} \label{mix-eq}
|\Puh(A \cap B) - \Puh(A) \Puh(B)| \leq 
2 (|U| + |V|) \,\, \frac{C_0}{\lfloor \al k \rfloor^{2+\ga}}.
\end{equation}
\end{lem}

\begin{proof}
Let $\hat A$ be the event $A \cap \{U \mbox{ is } \lfloor \al k \rfloor \mbox{ determined }\}$ \\
and $\hat B$ the event $B \cap \{V \mbox{ is } \lfloor \al k \rfloor \mbox{ determined }\}$. 
Note that, for each vertex $v$ and each integer $l$, the event that $v$ is $l$ determined depends
only on the random variables $X_{i_1(v)}, \cdots X_{i_l(v)}$. This, property (iii) in Section
2.1 and the fact that $U$ and $V$ have distance larger than $k$, implies that $\hat A$ and $\hat B$ are independent:

\begin{equation} \label{pf-mix-1}
\Puh(\hat A \cap \hat B) = \Puh(\hat A) \Puh(\hat B).
\end{equation}

Further, using \eqref{lconc},

\begin{equation} \label{pf-mix-2}
\Puh(A \setminus \hat A) \leq \Puh(U \mbox{ not } \lfloor \al k \rfloor \mbox{ determined })
\leq |U| \, \frac{C_0}{\lfloor \al k \rfloor^{2+\ga}},
\end{equation}
and similarly
\begin{equation} \label{pf-mix-3}
\Puh(B \setminus \hat B) \leq \Puh(V \mbox{ not } \lfloor \al k \rfloor \mbox{ determined })
\leq |V| \, \frac{C_0}{\lfloor \al k \rfloor^{2+\ga}}.
\end{equation}

From \eqref{pf-mix-1} - \eqref{pf-mix-3} Lemma \ref{mix-lem} follows straightforwardly.
\end{proof}
\end{subsection}

\begin{subsection}{Positive association} \label{sect-pa}

The next Lemma is about positive association.

\begin{lem}
\label{pa-lem}
The system $(\si_v, v\in \Z^2)$, described in Section 2.1, is positively associated.
That is, for all increasing (in terms of the $\si$ variables) events $A$ and $B$,
$\Puh(A \cap B) \geq \Puh(A) \Puh(B)$.
\end{lem}

\begin{proof}
\label{pa}
The random variables $(X_i, i \in I)$ are independent $\{0, 1, \cdots, k\}$ valued random variables
and hence (Harris-FKG inequality) positively associated. Since the $\si$ variables are increasing functions of the
$X$ variables, the statement of the lemma follows. 

\smallskip\noindent
{\bf Remark:} Note that the $\si_v, v \in \Z^2$ not necessarily satisfy the strong FKG condition.

\end{proof}

\end{subsection}
\begin{subsection}{RSW properties}
\label{rsw}
As said in the Introduction, Bollob\'as and Riordan (\cite{BoRi2}) obtained a new
RSW-like result for the
Voronoi percolation model. The conclusion of their RSW theorem is weaker than that of the classical
RSW theorem, but its proof is more robust: it does not (like the proof of `classical' RSW) involve conditioning
on the lowest crossing. It works, as they pointed out, not only for
the Voronoi model but for a large class of percolation models. In fact the conditions are as follows
(see \cite{BoRi3} and Section 4.3 in \cite{BeBrVa}) : \\
(a) Crossings of rectangles must be defined in terms of `geometric paths' in such a way 
that (for example)
horizontal and vertical crossings meet. (This enables the often used tool of pasting together
paths). \\
(b) Certain increasing events (in particular events of the form that there
is a $+$ path between two given sets of vertices) must be positively correlated. \\
(c) The distribution of the random field $(\si_v, v \in \Z^2)$ should be invariant under
the symmetries of $\Z^2$. \\
(d) Finally, certain mixing properties are needed.

\smallskip\noindent
The model in Theorem \ref{mainthm} satisfies the above conditions:
As to (a) these are simply well-known properties for percolation on the square lattice and
its matching lattice, and have nothing to do with the distribution $\Puh$. As to (b) and (c), these are taken care of by Lemma \ref{pa-lem} and by
property (iv) in Section 2.1 respectively. Finally, as to (d), 
the following property (here formulated in
our notation) is more than enough (see Remark 4.5 in \cite{BeBrVa}: 
For each $\ep >0$ there is an $l$ such that for all $k > l$, all $k$ by $2 k$ rectangles
$R_1$ and $R_2$ that have distance larger than $k/100$ to each other, and all events $A$ and
$B$ that are defined in terms of the random variables $(\si_v, v \in R_1)$ and in terms
of the random variables $(\si_v, v \in R_2)$ respectively, the following holds:
$|\Puh(A \cap B) - \Puh(A) \Puh(B)| < \ep$.
For our model this is immedially guaranteed by Lemma \ref{mix-lem}.
Hence our model belongs to the class of models mentioned above.

For this class of models the Bollob\'as-Riordan RSW like theorem (\cite{BoRi2} says that,
if the $\liminf$, as $s \ra \infty$,
of the probability of having a horizontal crossing of the box $[0,s] \times [0,s]$ is positive,
then, for every $\rho > 0$, the $\limsup_{s \ra \infty}$ of the probability of a horizontal crossing of
the box
$[0, \rho s] \times [0,s]$ is positive.

It is pointed out in \cite{BeBrVa} that small modifications
of the proof of \cite{BoRi2} give in fact the stronger result (for the same class of models as
described above) that if for {\it some}
$\rho > 0$ the $\limsup_{s \ra \infty}$ of the probability that there is a horizontal
crossing of the box $[0, \rho s] \times [0,s]$ is positive, then this holds for {\it all} $\rho >0$.
(Note the occurrence of $\limsup$ and $\liminf$).
Or, equivalently, if for some $\rho\,$ $\lim_{s \ra \infty}$ of the probability that there is a horizontal
crossing of the box $[0, \rho s] \times [0,s]$ equals $0$, then this limit equals $0$ for
every $\rho > 0$.
As remarked above, our current percolation model 
satisfies the required properties.
So we get the above mentioned RSW result. Before we state this explicitly we introduce the following notation.
Let $H(n,m)$ (respectively $V(n,m)$)  denote the event that there is a
horizontal (respectively vertical) $+$ crossing of the box $[0,n] \times [0,m]$.
Further, let $H^{- *}(n,m)$ and $V^{- *}(n,m)$ be the analogs of $H(n,m)$ and $V(n,m)$ for $-$
crossings in the $*$ lattice. In this notation the above mentioned RSW-like statement is:

\begin{lem}
\label{RSW-lem}
(a) If
$$\lim_{n \ra \infty} \Puh(H(\rho n, n)) = 0, \mbox{for some } \rho > 0,$$
then
$$\lim_{n \ra \infty} \Puh(H(\rho n, n)) = 0, \mbox{for all } \rho > 0.$$

\smallskip\noindent
(b) The analogous result, with $H$ replaced by $H^{- *}$, also holds.
\end{lem}

Note that, since a box either has a horizontal $+$ crossing or a vertical
$- *$ crossing (and using rotation symmetry), we have that for each $k$ and $l \,$, 
$\Puh(H(k,l)) = 1 - \Puh(H^{- *}(l,k))$. Combining this with Lemma \ref{RSW-lem} gives
immediately 

\begin{cor}
\label{RSW-cor} 

\smallskip\noindent
(a) If
$$\lim_{n \ra \infty} \Puh(H(\rho n, n)) = 1 \mbox{ for some } \rho > 0,$$
then
$$\lim_{n \ra \infty} \Puh(H(\rho n, n)) = 1 \mbox{ for all } \rho > 0.$$

\smallskip\noindent
(b) The analogous result, with $H$ replaced by $H^{- *}$, also holds.
\end{cor}

\end{subsection}
\begin{subsection}{Finite-size criterion}

\begin{lem}
\label{fs-lem}
 There is an $ \hat \ep > 0$ and an integer $\hat N$ such that for all $N \geq \hat N$ the
following holds: \\
\begin{equation}
\label{fs-cond}
\mbox{(a) If }\Puh(V(3 N, N)) < \hat\ep,
\end{equation}
then the distribution of $|C^+|$ has an exponential tail.

\smallskip\noindent
\begin{equation}
\label{fs-cond}
(b) If \Puh(V^{- *}(3 N, N)) < \hat\ep,
\end{equation}
then the distribution of $|C^{- *}|$ has an exponential tail.
\end{lem}

\begin{proof}
The proof below follows the main line of reasoning in the proof of the analogous well-known
result for ordinary percolation (see \cite{Ke81}).
Let $N$ and $\ep$ be such that $\Puh(V(3 N, N)) <\ep$. Cover $\Z^2$ by squares

$$Q_N(x) := N x + [0, N]^2, \, x \in \Z^2.$$
We will often write simply $Q_N$ for $Q_N(\bf 0)$. \\
We say that an $x \in \Z^2$ is {\it good}, if $Q_N(x)$ contains a vertex of $C^+$.
A set $W \subset \Z^2$ is called good if every $x \in W$ is good. Let 
$S$ denote the set of good vertices. From the
definition of `good' it is easy to see that $S$ is a connected subset of the square 
lattice, and that $0 \in S$ unless $C^+ = \emptyset$ (in which case also $S=\emptyset$).
It is also clear that $|S| \geq |C^+|/|Q_N|$, and hence that

\begin{equation}
\label{fsc-eq1}
\Puh\left(|C^+| \geq n\right) \leq \Puh\left(|S| \geq \frac{n}{|Q_N|}\right), \, n = 1, 2, \cdots .
\end{equation}

Let, for $x \in \Z^2$, $R_1(x)$ denote the $3 N \times N$ rectangle `north of 
$Q_N(x)$. More precisely 

$$R_1(x) := N x + [-N, 2N] \times [N, 2 N].$$

Similarly, let $R_2(x)$ be the $3N \times N$ rectangle south of $Q_N(x)$ and
let $R_3(x)$ and $R_4(x)$ be the $N \times 3 N$ rectangles east, respectively west of
$Q_N(x)$.

Define, for each $x \in \Z^2$,the following event (where `easy' stands for `vertical' in case of
a $3 N \times N$ reactangle, and for `horizontal' in case of an $N \times 3 N$ rectangle.

$$A_x := \{\exists i \in \{1, \cdots,4\} \mbox { s.t }
R_i(x) \mbox{ has a } + \mbox{ crossing in the easy direction }\}.$$

It is standard (and easy to check) that for all (except a finite number, say $C_1$)
$x \in \Z^2$ the following inclusion of events holds:

\begin{equation}
\label{fsc-eq2}
\{ x \mbox{ is good} \} \subset A_x.
\end{equation}

Let $R(x) = \cup_{i=1}^4 R_i(x)$. Recall the definition of `$l$ determined' in Section \ref{mix}.
We trivially have

\begin{eqnarray}
\label{fsc-eq3}
\, & \, & A_x \subset B_x, \,\, \mbox{ where } \\ \nonumber
\, & \mbox{where } & \\ \nonumber
\, & \, &B_x := \left(A_x \cap \{ R(x) \mbox { is } N \mbox{ determined }\}\right)
\cup \{R(x) \mbox{ is not } N \mbox{ determined }\}.
\end{eqnarray}

We then get
\begin{eqnarray}
\label{fsc-eq4}
\Puh(B_x) & \leq & \Puh(A_x) + \Puh(R(x) \mbox{ is not } N \mbox{ determined }) \\ \nonumber
\, &\leq& 4 \ep + |R(0)| \max_x \Puh(x \mbox{ is not } N \mbox{ determined }) \\ \nonumber
\, & \leq & 4 \ep + C_2 N^2 \frac{C_0}{N^{2 + \ga}} \\ \nonumber
\, &\leq & 4 \ep + C_3(N), \,\,\,\,\, \mbox{ where } C_3(N) \ra 0 \mbox{ as } N \ra \infty,
\end{eqnarray}
and where the first inequality is trivial, the second follows from our choice of $N$ and $\ep$,
the third follows from \eqref{lconc}, and $C_2$ is a constant.

Let $\al$ be as in property (iii) in Section 2.1. It is easy to see that there is a
constant $C_4 = C_4(\al)$ such that for every
finite set of vertices $x(1), \cdots, x(m)$ satisfying 
that $\min_{1 \leq i < j \leq m} ||x(i) - x(j)|| > C_4(\al)$, the events
$B_{x(i)},\, 1 \leq i \leq m$ are independent. \\

From this (and \eqref{fsc-eq2} - \eqref{fsc-eq4}) it follows easily that there are
a $C_5(\al)$  and $C_6(\al)$ such that for every finite set of vertices $W$,

\begin{eqnarray}
\label{fsc-eq5}
\Puh(W \mbox{ is good }) & \leq & (4 \ep + C_3(N))^{\lfloor(|W| - C_1)/C_5(\al)\rfloor}
\\ \nonumber
\, &\leq^*& (4 \ep + C_3(N))^{|W|/C_6(\al)},
\end{eqnarray}
where the mark * in the last inequality means that that inequality holds for all 
values of $|W|$ that are sufficiently large.

Now we apply this to \eqref{fsc-eq1}. To do this, note that if 
$|S| \geq n/|Q_N|$, there is a good lattice animal $W$ of size $\lfloor\frac{n}{|Q_N|}\rfloor$.
(A lattice animal is a connected set of vertices containing $0$). Using this, \eqref{fsc-eq1},
\eqref{fsc-eq5} and the fact that there is a constant $C_7$ such that the number of lattice sites
of size $k$ is at most $C_7^k$, we get

\begin{eqnarray}
\label{fsc-eq6}
\Puh(|C^+| \geq n) & \leq &
C_7^{\lfloor\frac{n}{|Q_N|}\rfloor} \, \left(4 \ep + C_3(N)\right)^{\lfloor\frac{n}{|Q_N|}\rfloor / C_6(\al)}
\\ \nonumber
\, & \leq & C_8(\ep,N) \,
\left[\left(C_7 \left(4 \ep + C_3(N)\right)^{\frac{1}{C_6(\al)}}\right)^{\frac{1}{|Q_N|}}\right]^n.
\end{eqnarray}
Now take $\hat \ep$ and $\hat N$ such that $C_7 (4 \hat\ep + C_3(\hat N))^{\frac{1}{C_6(\al)}} < 1$
for all $N \geq \hat N$
(which can be done since $C_3(N) \ra 0$ as $N \ra \infty$. 
From \eqref{fsc-eq6} it follows that for this choice of $\hat \ep$ and $\hat N$ the statement in part (a) of
Lemma \ref{fs-lem} holds. By exactly the same arguments (and, if necessary, by decreasing, respectively increasing
the values of $\hat \ep$ and $\hat N$ obtained above), part (b) also follows.
\end{proof}

From the above Lemma we easily get the following.

\begin{cor}
\label{fs-cor}
Let $\hat \ep$ and $\hat N$ as in Lemma \ref{fs-lem}. \\
(a)
If there is an $n \geq \hat N$ with $\Puh(V(3 n,n)) < \hat \ep$, then 
$\Puh(|C^{- *}| = \infty) > 0$.

\smallskip\noindent
(b) If there is an $n \geq \hat N$ with $\Puh(V^{- *}(3 n,n)) < \hat \ep$, then
$\Puh(|C^+| = \infty) > 0$.
\end{cor}

\begin{proof}
We only prove part (a) here; the proof of (b) is completely analogous.
If the condition of Corollary \ref{fs-cor} holds, then by Lemma \ref{fs-lem}, the distribution
of $|C^+|$ has an exponential tail. Exactly as in the Peierls argument in ordinary percolation (see e.g. \cite{Gr}),
this implies that the probability that there is a $+$ circuit having $0$ in its interior is
less than 1, and hence that $\Puh(|C^{- *}| = \infty) > 0$.
\end{proof}

\begin{lem}
\label{nce}
If $\Puh(|C^+| = \infty) > 0$, then $\Puh(|C^{- *}| = \infty) = 0$.
\end{lem}

\begin{proof}
There are various standard ways to prove this. One is as follows. The law of
$(\si_v, v \in \Z^2)$, is positively associated (by Lemma \ref{pa-lem}), translation invariant,
invariant under horizontal axis reflection and vertical axis reflection ((property
(iv) in section 2.1) and mixing (in the ergodic-theoretic sense, w.r.t. horizontal translations as well
as to vertical translations). The last follows from Lemma \ref{mix-lem}.
Hence, by the main result in \cite{GaKeRu}, 
$\Puh(|C^{- *}| = \infty) = 0$.
\end{proof}

\end{subsection}
\end{section}
\begin{section}{Proof of Theorem \ref{mainthm}}
We use the notation $\th(h)$ for $\Puh(|C^+| = \infty)$. \\
Let 
$$h_c := \sup\{h \, : \, \th(h) = 0\}.$$
It is quite easy to see that $h_c < \infty$: Take $n \geq \hat N$, with
$\hat N$ as defined in Lemma \ref{fs-lem}. From properties (b), (i) and (iii) in Section
2.1 it follows that for all $v \in \Z^2$, $\Puh(\si_v = +1) \ra 1$ as $h \ra \infty$, and
hence that $\Puh(H(3 n,n)) \ra 1$ as $h \ra \infty$, which is equivalent to
$\Puh(V^{- *}(3 n, n)) \ra 0$ as $h \ra \infty$. So there is an $h$ such that 
$\Puh(V^{- *}(3 n, n)) < \hat\ep$, with $\hat\ep$ as in Corollary \ref{fs-cor}. By part
(b) of that Corollary $\th(h) > 0$ for such $h$.
Hence, we have indeed that
$h_c < \infty$. With analogous arguments it follows that $h_c > -\infty$. \\

\smallskip\noindent
\begin{proof}
Now we start with the proof of part (a) of Theorem 2.2, where we will use the following notation.
$B(n)$ denotes the square $[-n,n]^2$, and $\partial B(n)$ its boundary (the set of all vertices $v$ that are
not in $B(n)$ but for which there is a $w \in B(n)$ with $||v-w|| = 1$). For $n \leq m$, $A(n,m)$ denotes
the annulus $B(m) \setminus B(n)$. For $v \in \Z^2$ and $n \in N$, $B(v;n)$ will denote the set $B(n)$ shifted
by $v$. 

Let $h$ be larger than the above defined
$h_c$. So $\Puh(|C^+| = \infty) > 0$. We will first show that

\begin{equation}
\label{sq1}
\Puh(H(n,n)) \ra 1 \mbox{ as } n \ra \infty.
\end{equation}

This is done in a quite standard way: Let $\de >0$.
Take $K$ so large that 

\begin{equation}
\label{K-cond}
\Puh(B(K) \ra \infty) > 1 - \de.
\end{equation}
By Lemma \ref{nce} we can take $N > K$ so large that

\begin{equation}
\label{circNK}
\Puh(\exists \mbox { a } + \mbox{ circuit in } A(K,N) \mbox{ surrounding } B(K)) > 1 - \de.
\end{equation}

For all $n \geq N$ the following holds. First,
by \eqref{K-cond} we have, of course, that $\Puh(B(K) \lra \partial B(n)) > 1 - \de$. Since our
model has the positive association property (see Lemma \ref{pa-lem}) we can apply the usual `square root trick'
(see e.g. \cite{Gr}),
which gives that $\Puh(B(K) \lra r(B(n))) > 1 - \de^{1/4}$, where
$r(B(n))$ stands for the right side $\{n\} \times [-n,n]$ of $B(n)$. By this, and its analog for the left
side $l(B(n))$ of $B(n)$,
together with \eqref{circNK} (and again positive association) we get, for all $n \geq N$,

\begin{eqnarray}
\, & \, & \Puh(H(n,n))  \\ \nonumber 
\, & \geq & \Puh( B(K) \lra r(B(n)), \, B(K) \lra l(B(n)), + \mbox{ circuit in } A(K,n)) \\ \nonumber
\, & \geq &  (1 - \de^{1/4})^2 (1 - \de).
\end{eqnarray}
Since we can take $\de$ arbitrary small, \eqref{sq1} follows.

Application of Corollary \ref{RSW-cor} now gives that 
$\Puh(H(3 n, n)) \ra 1$, as $n \ra \infty$, and hence that 
$\Puh(V^{- *}(3 n,n)) \ra 0$ as $n \ra \infty$.

Finally,  by part (b) of Lemma \ref{fs-lem} this implies that the distribution of
$|C^{- *}|$ has an exponential tail. This completes the proof of Theorem 2.2 (a).

It is important to note that part (b) of the Theorem can not simply be concluded 
by replacing $`+'$ by $`- *$ (and v.v.) in the arguments above. The problem is that our definition of
$h_c$ in the beginning of the proof is `asymmetric'. If we could show that the above defined $h_c$ is equal to
$\inf\{h \, : \,  \Puh(|C^{- *}| = \infty) = 0\}$, or equivalently (since we already know by 
Lemma \ref{nce} that there is no $h$ for which both $\th(h) >0$ and $\th^{- *}(h) >0$) that 
$\th^{- *}(h) > 0$ for all $h < h_c$, we would be able to conclude (b) by exchanging $+$ and $- *$ in the
arguments of (a).
Below it will be shown, using the approximate zero-one laws in Section \ref{sect-azol}, that
indeed $\th^{- *}(h) > 0$ for all $h < h_c$.

\smallskip\noindent
{\it Proof of (b)} 

\smallskip\noindent
Suppose there is a $h_1 < h_c$ with 
$\thsm(h_1) = 0$. We will show that this leads to a contradiction.
Let $h_2 \in (h_1, h_c)$. Then, for all $h \in [h_1, h_2]$, by monotonicity (see properties (a) and (i) in
Section \ref{term})
$\th(h) = \thsm(h) = 0$.
Let $H(n,m)$  and $\Hsm(n,m)$ be the box crossing events defined in 
Section \ref{rsw}.
Since $\thp \equiv 0$ on $[h_1,h_2]$, we have, by Corollary \ref{fs-cor} (b) (noting that
$\Puh(V^{- *}(3 n,n)) = 1 - \Puh(H(3 n,n))$ that

\begin{equation}
\label{fsc-b}
\forall h \in [h_1,h_2] \,\, \forall n \geq \hat N, \,\, P^{(h)}(H(3 n,n)) < 1 - \hat\ep, 
\end{equation}
with $\hat \ep$ and $\hat N$ as in Lemma \ref{fs-lem}. \\
On the other hand, $P^{(h_1)}(H(n, 3 n)) = P^{(h_1)}(V(3 n, n))$, which (again by Corollary \ref{fs-cor}, and
because $\th^{- *}(h_1) = 0$) 
is at least $\hat \ep$ for all $n \geq \hat N$.
Hence, by Lemma \ref{RSW-lem},

$$\limsup_{n \ra \infty} P^{(h_1)}(H(3 n, n)) >0.$$
Using this, monotonicity and \eqref{fsc-b} it follows straightforwardly that there is a $\de \in (0,1)$
and an infinite sequence $n_1 < n_2 < n_3 \cdots$, such that

\begin{equation}
\label{contr-1}
\Puh(H(3 n_i,n_i)) \in (\de, 1-\de), \mbox{ for all } i \mbox{ and all } h \in [h_1,h_2].
\end{equation}

To reach a contradiction, we will show that the sequence $(\ep_n)$, defined by
$$\ep_n := \sup_{j \in I,\,  h \in [h_1,h_2]} \Puh\left((H(3 n,n))_j\right)$$
satisfies

\begin{equation}
\label{contr-2}
\ep_n \ra 0, \,\, \mbox{ as } n \ra 0,
\end{equation}

where, as before (Section 3.1), $A_j$ denotes the event that $j$ is an internal pivotal for the event
$A$.  

\smallskip\noindent
{\bf Remark:}
It is important to note that here we do not (as in ordinary percolation and in Higuchi's treatment)
consider pivotality in terms of the vertices of the lattice (the indices of the $\si$ variables), but in
terms of the indices of the underlying
$X$ variables (that is, in the special case of the Ising model, the space-time variables
$Y_v(t)$ in Section \ref{spc}, which control the updates in the dynamics).

\smallskip\noindent
We will first show that \eqref{contr-1} and \eqref{contr-2} indeed give a contradiction:
First define, for all positive integers $n, m, N$, the event
$$H(n,m)(N):= H(n,m) \cap \{[0,n] \times [0,m] \mbox{ is } N \mbox{ determined }\}.$$ 
It is easy to see from the definitions that $H(n,m)(N)$ is increasing, and (by \eqref{lconc})
that for fixed $n, m$ the sequence $H(n,m)(N)$, $N = 1, 2, \cdots$ strongly approximates the
event $H(n,m)$ in the sense of Definition \ref{def-appr}.
By Corollary \ref{tal1.3xx} we thus have, for all $i = 1, 2, \cdots,$

$$P^{(h_1)}(H(3 n_i, n_i)) \left(1 - P^{(h_2)}(H(3 n_i, n_i))\right) \leq 
\ep_{n_i}^\frac{c(h_1,h_2) (h_2 - h_1)}{K_2},$$
with 
$\ep_{n_i}$ as defined above.

By \eqref{contr-2}, the r.h.s. in this last inequality goes to $0$ as $n \ra \infty$. However, for
all $i$ the l.h.s. is at least $\de^2$ by \eqref{contr-1}: a contradiction.

So part (b) of the theorem is proved once we prove \eqref{contr-2}, which we will do now:
In the following, $X$ stands for the collection of random variables $(X_i, i \in I)$.

Note that, by the definition of internal pivotal,

\begin{equation}
\label{contr2-1}
\Puh\left((H(3 n,n))_j \right) = P\left(X \in H(3 n, n), \, X^{(j)} \not\in H(3 n, n)\right),
\end{equation}
where $X^{(j)}$ is the element of $\{0, \cdots,k\}^I$ that satisfies 
$X^{(j)}_i = X_i$ for all $i \neq j$ and $X^{(j)}_j = 0$.

Now recall that for each $v \in \Z^2$ we have the sequence $i_1(v), i_2(v), \cdots$ 
introduced in property (ii) in Section 2.1. We will use the following terminology. 
If $i_m(v) = j$, we say that $j$ has rank $m$ for $v$. If $j$ does not occur at all in
the sequence $i_1(v), i_2(v), \cdots$, we say that the rank of $j$ for $v$ is infinite.
The rank of $j$ for $v$ will be denoted by $r_v(j)$.
Suppose that $r_v(j) = m$. Then we say that $v$ {\it needs} $j$ if
$(X_{i_1(v)}, X_{i_2(v)}, \cdots, X_{i_{m-1}}(v))$ does not determine $\si_v$.

Let $v$ be a vertex in the box $[0, 3 n] \times [0,n]$.
We use the notation $H(3 n,n;\, v)$ for the event that $v$ is on a horizontal $+$ crossing of
that box.
Using the terminology and observation above, we have that
the r.h.s. of \eqref{contr2-1} is at most 

\begin{eqnarray}
\label{contr2-2}
\, & \, &  P\left(\exists v \in \Z^2 \mbox{ s.t. } X \in H(3 n,n;\, v), \mbox{ but } 
X^{(j)} \not \in H(3 n,n;\, v)\right) \\ \nonumber
\, & \leq & \sum_{v \in \Z^2} P(X \in H(3 n,n;\, v),\, X^{(j)} \not \in H(3 n,n;\, v))
\\ \nonumber 
\, & \leq & \sum_{v \in \Z^2} \Puh( H(3 n,n;\, v), \, v \mbox{ needs } j) \\ \nonumber
\, & \leq & \sum_{v \in \Z^2} \min(\Puh(H(3 n,n;\, v)), \, \Puh(v \mbox{ needs } j)).
\end{eqnarray}

Further, note that if $v$ is on a horizontal $+$ crossing of the rectangle
$[0,3 n] \times [0, n]$, there must be a $+$ path from $v$ to $\partial B(v;n)$.
By this, and translation invariance (property (iv) in Section
2.1), the first of the two probabilities in the expression in the summand
in the last line of \eqref{contr2-2} (that is, $\Puh(H(3 n, n; v))$) is at most
$\Puh(0 \lra \partial B(n))$, which by monotonicity is of course at most 
$P^{(h_2)}(0 \lra \partial B(n))$. Let us denote this last probability by $f(n)$.
Also note that property (ii) in Section 2.1 says that

$$\Puh(v \mbox{ needs } j) \leq \frac{C_0}{(r_v(j) - 1)^{2+ \ga}}.$$

These considerations imply that the last line of \eqref{contr2-2}
is, for each positive integer $K$, at most

\begin{equation}
\label{contr2-3}
f(n) \times |\{v \in \Z^2 \, : \, r_v(j) \leq K \}| 
+ C_0 \sum_{k = K}^{\infty} \frac{|\{v \, : \, r_v(j) = k\}|}{(k-1)^{2 + \ga}}.
\end{equation}

Consider the set in the first term in \eqref{contr2-3}. Let $u$ and $w$ be two vertices which both belong to this set. That is, $r_u(j) \leq K$ and $r_w(j) \leq K$ hold, and hence
the sets $\{i_1(u), \cdots, i_K(u)\}$ and $\{i_1(w), \cdots, i_K(w)\}$ have non-empty
intersection. It follows from property (iii) in Section 2.1 that $||v - w||$ is 
at most $K/\alpha$. Hence, the set under consideration has diameter $\leq K/\al$, and 
hence the cardinality of this set satisfies

\begin{equation}
\label{card-ineq}
\vert \{v \in \Z^2 \, : \, r_v(j) \leq K \}\vert \leq \frac{C_9 K^2}{\al^2},
\end{equation}
for some constant $C_9$.

From this (and using that $(k-1)^{2 + \ga}$ is decreasing in $k$) it is easy to see that the sum in 
\eqref{contr2-3} satisfies

\begin{equation}
\label{extr-eq}
\sum_{k = K}^{\infty} \frac{|\{v \, : \, r_v(j) = k\}|}{(k-1)^{2 + \ga}} \leq  \frac{C_{10}}{\al^2 K^\ga}
+ \sum_{k = K + 1}^{\infty} \frac{C_{10}}{\al^2 k^{1 + \ga}},
\end{equation}
for some constant $C_{10}$.

Note that in \eqref{contr2-3} we are free to choose $K$. In the following we let 
$K(n)$ be the largest integer $k$ for which

$$\frac{C_9 k^2}{\al^2} \leq \frac{1}{\sqrt{f(n)}}.$$

Taking together \eqref{contr2-1} - \eqref{extr-eq} we get, choosing $K = K(n)$ in
\eqref{contr2-3},

\begin{equation}
\label{contr2-4}
\Puh\left((H(3 n,n))_j\right)  \leq   f(n) \frac{1}{\sqrt{f(n)}}  + \frac{C_{10}}{\al^2 K(n)^{\ga}} 
+ \sum_{k = K(n) + 1}^{\infty} \frac{C_{10}}{\al^2 k^{1 + \ga}}.
\end{equation} 

Note that the r.h.s. of \eqref{contr2-4} does not depend on $j$ and $h$, and
(since $f(n) \ra 0$ as $n \ra \infty$, $\ga > 0$, and $K(n) \ra \infty$ as $n \ra \infty$)
goes to $0$ as $n \ra \infty$. 
This proves \eqref{contr-2}, and thus
completes the proof of the first statement in  part (b) of the Theorem. The second statement of part (b) follows
now in exactly the same way as its analog in (a).
\end{proof}

\end{section}
{\large \bf Acknowledgments} \\
My interest in some of the problems in this paper was raised by communication with
Frank Redig and Federico Camia. I also thank Rapha\"el Rossignol for his valuable
comments on Corollary \ref{tal1.3xx}.


\begin{thebibliography}{A}
\bibitem[MaOl]{MaOl} Martinelli Olivieri paper that spatial mixing implies temporal mixing
\end{thebibliography}
\end{document}